\documentclass[11pt]{article}

\usepackage{times,amsthm}
\usepackage{amsmath,amssymb}
\usepackage{mathrsfs}
\usepackage{verbatim}
\usepackage{fullpage}
\usepackage{textcomp}

\usepackage{titlesec}
\titleformat{\section}{\large\scshape\centering}{\thesection.}{1em}{}
\titleformat{\subsection}{\normalsize\bfseries}{\thesubsection}{1em}{}

\newtheorem{theorem}{Theorem}[section]
\newtheorem{lemma}[theorem]{Lemma}
\newtheorem{definition}{Definition}
\newtheorem{proposition}{Proposition}
\newtheorem{corollary}{Corollary}

\newtheorem{remark}{Remark}
\newtheorem{example}{Example}

\newcommand{\N}{\mathbb{N}}
\newcommand{\Z}{\mathbb{Z}}
\newcommand{\R}{\mathbb{R}}
\newcommand{\C}{\mathbb{C}}

\newcommand{\bit}{\text{bit}}

\newcommand{\famlyF}{\mathscr{F}}
\newcommand{\famlyG}{\mathscr{G}}

\begin{document}
\pagestyle{plain} 
\lineskip=1.8pt\baselineskip=18pt\lineskiplimit=0pt 

\title{On Sets of Large Fourier Transform Under Changes in Domain}

\author{Joel Laity and Barak Shani \\ \small{Department of Mathematics, University of Auckland, New Zealand}}

\date{}
\maketitle

\begin{abstract}
A function $f:\Z_n \to\C$ can be represented as a linear combination $f(x)=\sum_{\alpha \in \Z_n}\widehat{f}(\alpha) \chi_{\alpha,n}(x)$ where $\widehat{f}$ is the (discrete) Fourier transform of $f$. Clearly, the basis $\{\chi_{\alpha,n}(x):=\exp(2\pi i \alpha x/n)\}$ depends on the value $n$.

We show that if $f$ has ``large'' Fourier coefficients, then the function $\widetilde{f}:\Z_m \to \C$, given by
\[
	\widetilde{f}(x)
	= \begin{cases}
		f(x)     	& \text{when } 0\leq x < \min(n, m), \\
		0			& \text{otherwise},
	\end{cases}
\]
also has ``large'' coefficients.
Moreover, they are all contained in a ``small'' interval around $\lfloor \frac{m}{n}\alpha \rceil$ for each $\alpha \in \Z_n$ such that $\widehat{f}(\alpha)$ is large.
One can use this result to recover the large Fourier coefficients of a function $f$ by redefining it on a convenient domain. One can also use this result to reprove a result by Morillo and R{\`a}fols: \emph{single-bit} functions, defined over any domain, have a small set of large coefficients.
\end{abstract}

\section{Introduction}

In recent years different areas of mathematics, such as additive number theory \cite{Green}, combinatorial number theory \cite{Shkredov} and cryptography \cite{MR}, have seen results that take advantage of the structure arising from large values of the Fourier transform.
Most notably, the quantum algorithm for period finding given by Shor \cite{Shor} is an application that exploits this structure.

In this paper we are concerned with what happens to the Fourier coefficients of a function when we extend or restrict its domain.
Our convention is to set the function to be zero on the new added values.

\begin{definition}
Let $m, n\in \N$.
Let $f:\Z_n\to \C$.
Define $\widetilde{f}:\Z_m\to \C$ by
\[
	\widetilde{f}(x)
	= \begin{cases}
		f(x)     	& \text{when } 0\leq x < \min(n, m), \\
		0			& \text{otherwise}.
	\end{cases}
\]
\end{definition}

In this paper we will establish a relationship between the large Fourier coefficients of $f$ and $\widetilde{f}$.
Our results are motivated by the following:
\begin{enumerate}
  \item When $m>n$, our results have some relevance to approximations of functions.
  \\
  There is a broad area of research, mostly in the computer science and engineering communities, that studies algorithms to recover the set of points for which the Fourier transform of a given function is large (see \cite{Fourier_survey} for a recent survey).
  Such algorithms have been extensively studied and optimized for domains of size a power of $2$.
  \\
  Therefore, our motivation is to show that if the large Fourier coefficients of $\widetilde{f}$ can be recovered for $m=2^k$, one can then determine the set of large Fourier coefficients of $f$.

  \item When $m<n$, our results apply to functions that can be naturally defined over any domain and the families of functions they give rise to $\{f_n:\Z_n\to\C\}_{n\in\N}$.
  \\
  One example of such a family is to take a periodic function $f:\Z\to\C$ and restrict its domain by defining $f_n(x)=f(x)$ for all $x\in \Z_n$. Another example is the family of most-significant-bit (MSB) functions: for $2^k<n\leq2^{k+1}$, MSB$(x)=0$ if $x<2^k$ and MSB$(x)=1$ otherwise. It turns out that the analysis of the Fourier transform for these functions is much easier on some domains than others.
  \\
  Thus, our motivation is to show that by analyzing the large Fourier coefficients of only a small proportion of functions in the family, one can analyze the large Fourier coefficients of all the functions in the family.
\end{enumerate}

\subsection{Definitions and Results}

We first give some definitions as well as some basic properties of the Fourier transform.

We use $\lfloor x \rceil$ to denote the closest integer to $x\in \R$ and we use $\|x\|$ for the distance between $x$ and its closest integer, i.e. $\|x\|=| \lfloor x \rceil - x |$. Let $k \in \N$, for all $x \in \R$ define $|x|_k = \min \{|x-kz| \mid z\in \Z\}$, so $|x|_k$ is the distance from $x$ to the nearest integer multiple of $k$. Equivalently we can define $|x|_k = k\| x/k \| $.

The group $\Z_n$ is the set $\{0, 1, \dots , n-1\}$ under the group operation of addition modulo $n$. The set $L^2(\Z_n) = \{f:\Z_n\to \C\}$ of all functions from $\Z_n$ to $\C$ forms a vector space under the usual addition and scalar multiplication of functions.
This space has the inner product $\langle f, g\rangle = \frac{1}{n}\sum_{x\in \Z_n}f(x)\overline{g(x)}$, where $\overline{z}$ denotes the complex conjugate of $z\in \C$,  and norm $\|f\|_2 = \sqrt{\langle f, f\rangle}$.

For each $\alpha\in \Z_n$ the additive character $\chi_{\alpha,n}$ is given by $\chi_{\alpha, n}(x) = \omega_n^{\alpha x}$, where $\omega_n = \exp(2\pi i /n)$ is the $n$-th root of unity.
When it is clear from context what the domain of a character is, we will omit the second index and use $\chi_\alpha$ to denote $\chi_{\alpha , n}$.
The set of all characters $\{\chi_{\alpha}\}_{\alpha\in \Z_n}$ forms an orthonormal basis for $L^2(\Z_n)$; thus any function $f:\Z_n\to \C$ can be written as
\begin{equation} \label{eq:fourier equation}
	f(x) 		= 		\sum_{\alpha\in\Z_n}\widehat{f}(\alpha)\chi_{\alpha}(x) \, ,
\end{equation}
 where $\widehat{f}(\alpha)\in \C$ is the $\alpha$-th \emph{Fourier coefficient}.
 Equation (\ref{eq:fourier equation}) implicitly defines a function $\widehat{f}:\Z_n\to \C$, called the Fourier transform of $f$.
 It can be explicitly calculated using the formula
\[
	\widehat{f}(\alpha)
	= \langle f, \chi_{\alpha} \rangle
	= \frac{1}{n}\sum_{x\in\Z_n}f(x)\overline{\chi_{\alpha}(x)} \, .
\]
One of the most fundamental identities in Fourier analysis, and one used throughout this paper, is Parseval's identity:
\[
	\frac{1}{n}\sum_{x\in \Z_n}|f(x)|^2
	= \|f\|_2^2
	= 	\sum_{\alpha \in \Z_n}|\widehat{f}(\alpha)|^2 \, .
\]

Given a function $f:\Z_n\to \C$ and a set $\Gamma \subseteq \Z_n$ we use the notation $f|_\Gamma$ to denote the projection of $f$ onto the subspace $\operatorname{span}\{\chi_{\alpha} \mid \alpha \in \Gamma\}\subseteq L^2(\Z_n)$; that is $f|_\Gamma$ is a function from $\Z_n$ to $\C$ and $f|_\Gamma(x) := \sum_{\alpha \in \Gamma}\widehat{f}(\alpha) \chi_{\alpha}(x)$.
Note that this definition of $f|_\Gamma$ is different from the usual meaning of the notation.

A function is \emph{$k$-sparse} if $\widehat{f}(\alpha)$ is nonzero for at most $k$ values of $\alpha \in \Z_n$.
In this paper we are concerned with two types of functions.
The first is of functions with Fourier coefficients whose magnitude is large \emph{relative to the function's norm}.
The second type is of functions that can be approximated, up to any error term, in the form $f|_\Gamma$ where $\Gamma$ has small cardinality.
We call these functions concentrated.

\begin{definition} \label{def:eps-concentration}
	Let $\{n_k\}_{k\in \N}$ be a sequence of positive integers.
	Let $\epsilon > 0$.
	A family of functions $\{f_k:\Z_{n_k}\to \C \}_{k\in \N}$ is \emph{$\epsilon$-concentrated} if there exists a polynomial $P\in \R[x]$ such that for all $k\in \N$ there exists a set $\Gamma_k\subseteq \Z_{n_k}$ where $|\Gamma_k|\leq P(\log(n_k))$ and $\|f_k-f_k |_{\Gamma_k} \|_2^2 <\epsilon$.
\end{definition}

\begin{definition} \label{def:concentration}
	Let $\{n_k\}_{k\in \N}$ be a sequence of positive integers.
	A family of functions $\{f_k:\Z_{n_k}\to \C \}_{k\in \N}$ is \emph{concentrated} if there exists a bivariate polynomial $P\in \R[x, y]$ such that for all $k\in \N$ and for all $\epsilon > 0$ there exists a set $\Gamma_k\subseteq \Z_{n_k}$ where $|\Gamma_k|\leq P(\log(n_k),1/\epsilon)$ and $\|f_k-f_k |_{\Gamma_k} \|_2^2 <\epsilon$.
\end{definition}

Roughly speaking, Definition \ref{def:concentration} says that a family of functions is concentrated if it can be approximated up to any error term by polynomially many characters (with respect to the domain size and the approximation error).

Most applications are concerned with a single function that implicitly defines the entire family for all possible domains (that is, $n_k=k$ or $n_k=c+k$ for some constant $c$).
In this case we abuse notation and say that the function, instead of the family, is concentrated.
Here are some examples.

\begin{example}\label{ex:concentrted}~

\begin{itemize}
  \item A single character defines a concentrated family. More precisely if we fix $\alpha\in \N$ the family of functions $ \{\chi_{\alpha}:\Z_n \to \C\}_{n>\alpha}$ is concentrated.
 	For every integer $n>\alpha$ we can choose $\Gamma_n = \{\alpha\}$, then $\|\chi_{\alpha} - \chi_{\alpha}|_{\Gamma_n}\|_2^2 = 0$ and hence the definition is satisfied with the constant polynomial $P=1$.

  \item Single-bit functions on domains of size powers of $2$ are concentrated families. More precisely, let $k \in \N$ and $0 \le i < k$. Define $\bit_i : \Z_{2^k} \to \{-1,1\}$ by $\bit_i(x) = (-1)^{x_i}$ where $x = \sum_{j=0}^{k-1} x_j 2^j$ and $x_j \in \{ 0,1 \}$. Then, as we show in Lemma~\ref{lemma:bits} below, the family $\mathscr{F}_i = \{\bit_i:\Z_{2^k} \to \{-1,1\}\}_{k>i}$ is concentrated for every $i\in \N$.

      Moreover, as a corollary of our result in Section \ref{sec:conc_proof}, we show that single-bit functions are concentrated over every domain.

  \item The family $\{f_n:\Z_n\to \C\}_{n\in \N}$ where $\widehat{f}_n(\alpha) = 1$ and $\widehat{f}_n(\alpha') = 1/\sqrt{n}$ for all $\alpha' \neq \alpha$ is $\epsilon$-concentrated for $\epsilon = 1$.
       We can choose $\Gamma_n = \{\alpha\}$ for every $n\in \N$ then $\|f_n - f_n|_{\Gamma_n}\|_2^2 = (n-1)(\frac{1}{\sqrt{n}})^2 = 1-1/n < \epsilon$.
       However this family is not $\epsilon$-concentrated for any $\epsilon < 1$ since even if $\alpha \in \Gamma_n$ we have $\|f_n-f_n|_{\Gamma_n}\|_2^2 = \sum_{\alpha' \in\Z_n, \alpha' \not \in \Gamma_n}|\widehat{f}_n(\alpha')|^2 = (n-|\Gamma_n|)(\frac{1}{\sqrt{n}})^2 = 1-|\Gamma_n|/n$.
       Hence if $|\Gamma_n|$ is polynomial in $\log(n)$ then $|\Gamma_n|/n < 1-\epsilon$ for sufficiently large $n$, so $\|f_n-f_n|_{\Gamma_n}\|_2^2 = 1-|\Gamma_n|/n > \epsilon$.

       Functions with similar Fourier transform to this example arose for example from `noisy' characters, given by $f(x):=\omega_n^{\alpha x+e(x)}$ for some suitable random bounded functions $e$.
\end{itemize}
\end{example}

In this work, we study how the large Fourier coefficients of the function $f$ are related to those of $\widetilde{f}$.
It is clear that when $m$ is much smaller than $n$, we lose a lot of information about the function $f$ (as shown in Example \ref{ex:switch_down} below).
In order to avoid extreme cases of unbalanced functions\footnote{An example of such extreme case is a function where $f(0)=n$ and otherwise $f(x)=0$; then $\|f\|_2^2=1$, and if one takes $m=1$ then $\|\widetilde{f}\|_2^2=n$.}, in our main theorems we restrict the domain $\Z_m$ such that $m \geq n/2$.
We remark that our results also hold when $m \geq n/poly(\log(n))$.\footnote{Notice that by taking $m \gg n$ some of the results hold trivially.}

Our first result treats concentrated families.

\begin{theorem}\label{thm:concentrated}
Let $\{n_k\}_{k\in\N},\{m_k\}_{k\in\N}$ two sequences of positive integers with $m_k \geq n_k/2$ for every $k\in\N$.
Let $Q\in \R[x]$ be a polynomial.
Let $\{f_k:\Z_{n_k}\to\C\}_{k\in\N}$ be a concentrated family of functions such that $\|f_k\|_2^2 \leq Q(\log(n_k))$ for all $k\in \N$.
Then $\{\widetilde{f_k}:\Z_{m_k}\to\C\}_{k\in\N}$ is a concentrated family of functions.
\end{theorem}

Furthermore, we show that there is a simple relationship between the sets $\Gamma_k$ on which $f_k$ are concentrated and the sets on which $\widetilde{f_k}$ are concentrated.
This leads to an algorithmic approach to approximate a function $f$, which we treat in Section~\ref{sec:algorithms}: by finding an approximation of $\widetilde{f}$ with a sparse function, one can translate it to approximate $f$ with a sparse function.

Our second result treats any function $f$ with large Fourier coefficients.

\begin{theorem}\label{thm:e-concentrated}
Let $\{n_k\}_{k\in\N},\{m_k\}_{k\in\N}$ two sequences of positive integers with $m_k \geq n_k/2$ for every $k\in\N$.
Let $t:= \sup_{k\in\N}\{n_k/m_k\}$.
Let $Q\in \R[x]$ be a polynomial.
Let $\{f_k:\Z_{n_k}\to\C\}_{k\in\N}$ be a family of functions such that $\|f_k\|_2^2 \leq Q(\log(n_k))$ for all $k\in \N$.
Then if $\{f_k:\Z_{n_k}\to\C\}_{k\in\N}$ is an $\epsilon$-concentrated family of functions, the family $\{\widetilde{f_k}:\Z_{m_k}\to\C\}_{k\in\N}$ is a $(t\epsilon+\eta)$-concentrated family of functions for any $\eta>0$.
\end{theorem}

\section{Preparation}
We first show how to express the Fourier transform of $\widetilde{f}$ in terms of the Fourier coefficients of $f$.

\begin{lemma}\label{lemma:ghat}
Let $m, n\in \N$.
Let $f:\Z_n\to \C$.
Let $\ell := \min(n,m)$.
Let $g:\Z_m\to \C$ be defined by $g(x) = f(x) $ for $0 \leq x < \ell$ and $g(x) = 0$ otherwise.
Then
\begin{equation} \label{eq:ghat}
	\widehat{g}(\beta)
	= \sum_{\alpha \in \Z_n} \widehat{f}(\alpha) \frac{1}{m} \sum_{x \in \Z_\ell} \omega_{m}^{(\frac{m}{n}\alpha-\beta) x}
\end{equation}
for all $\beta \in \Z_m$.
\end{lemma}

\begin{proof}
Recall that $f(x) = \sum_{\alpha \in \Z_n} \widehat{f}(\alpha) \chi_{\alpha }(x)$. So
\[
	\begin{split}
		\widehat{g}(\beta)
		& = \frac{1}{m} \sum_{x \in \Z_m} g(x) \overline{\omega_m^{\beta x}}
		= \frac{1}{m} \sum_{x \in \Z_\ell} f(x) \omega_m^{-\beta x} \\
		& = \frac{1}{m} \sum_{x \in \Z_\ell} \sum_{\alpha \in \Z_n} \widehat{f}(\alpha) \omega_n^{\alpha x} \omega_m^{-\beta x}
		=  \sum_{\alpha \in \Z_n} \widehat{f}(\alpha) \frac{1}{m} \sum_{x \in \Z_\ell} \omega_{m}^{(\frac{m}{n}\alpha-\beta) x} \, .
	\end{split}
\]
\end{proof}

Lemma \ref{lemma:ghat} tells us that the Fourier coefficient $\widehat{g}(\beta)$ is a weighted sum of the Fourier coefficients of $f$.
The weights are given by the exponential sum $\frac{1}{m} \sum_{x \in \Z_\ell} \omega_{m}^{(\frac{m}{n}\alpha-\beta) x}$, the magnitude of which is inversely proportional to the distance between $\frac{m}{n}\alpha$ and $\beta$ (we will give a precise statement and proof in Lemma \ref{lemma:bound_character}).
Hence if $\beta$ is very far from $\frac{m}{n}\alpha$ then the $\widehat{f}(\alpha)$-term in Equation (\ref{eq:ghat}) will contribute very little to $\widehat{g}(\beta)$.

\begin{lemma}\label{lemma:tilde_properties}
Let $m, n\in \N$.
The mapping $\widetilde{\cdot}:L^2(\Z_n)\to L^2(\Z_m)$ is linear.
Moreover,
\[
	\langle \widetilde{f}, \widetilde{g}\rangle 	= 	\frac{n}{m}\langle f, g\rangle
	\qquad \text{ if $n<m$, \quad and } \qquad
	\|\widetilde{f}\|_2^2 	\leq 	\frac{n}{m}\|f\|_2^2
\]
for all $f, g:\Z_n\to \C$, with equality if $n<m$.
\end{lemma}

\begin{proof}
The linearity follows directly from the definition of $\, \widetilde{\cdot} \,$.
From the definition of the inner product $\langle \widetilde{f}, \widetilde{g}\rangle = \frac{1}{m}\sum_{x = 0}^{m-1}\widetilde{f}(x)\overline{\widetilde{g}(x)} = \frac{1}{m}\sum_{x=0}^{n-1}f(x)\overline{g(x)} = \frac{n}{m}\langle f, g \rangle$, when $n<m$. In addition, we always have $\langle \widetilde{f}, \widetilde{f}\rangle = \frac{1}{m}\sum_{x = 0}^{m-1}\widetilde{f}(x)\overline{\widetilde{f}(x)} \leq \frac{1}{m}\sum_{x=0}^{n-1}f(x)\overline{f(x)} = \frac{n}{m}\langle f, f \rangle$.
The last property now follows since $\|\widetilde{f}\|_2^2 = \langle \widetilde{f}, \widetilde{f}\rangle \leq \frac{n}{m}\langle f, f \rangle = \frac{n}{m}\|f\|_2^2$.
\end{proof}

\begin{remark} \label{rem:inner product}
Another version of Lemma \ref{lemma:ghat} is
\[
\widehat{g}(\beta) = \sum_{\alpha \in \Z_n} \widehat{f}(\alpha) \langle \widetilde{\chi}_{\alpha, n} , \chi_{\beta, m} \rangle \, .
\]
We can also give an alternative proof that uses the properties of $\; \widetilde{\cdot} \, $ instead of explicit calculations.
From the linearity of $\; \widetilde{\cdot}\, $ we get $g = \widetilde{f} = \sum_{\alpha \in \Z_n} \widehat{f}(\alpha) \widetilde{\chi}_{\alpha,n}$.
Then \[\widehat{g}(\beta) = \langle g, \chi_{\beta, m}\rangle  = \langle \sum_{\alpha \in \Z_n}\widehat{f}(\alpha)\widetilde{\chi}_{\alpha,n}, \chi_{\beta, m}\rangle = \sum_{\alpha \in \Z_n} \widehat{f}(\alpha) \langle \widetilde{\chi}_{\alpha, n} , \chi_{\beta, m} \rangle\,.\]
Thus $\widehat{g}(\beta)$ is weighted sum of the Fourier coefficients of $f$ where the weights are $\langle \widetilde{\chi}_{\alpha, n} , \chi_{\beta, m} \rangle$.
\end{remark}

In order to bound the exponential sum in Equation (\ref{eq:ghat}) we will need the following lemma.
\begin{lemma}\label{lemma:weyl}
Let $N\in \N$.
Let $x \in \R\setminus{\Z}$.
Let $e(x):=\exp(2\pi i x)$.
Then
\[
	\sum_{k=0}^{N} e(x k) 	= 		e\left(x N /2\right) \frac{\sin(\pi x (N+1))}{\sin(\pi x)} \, .
\]
\end{lemma}

\begin{proof}
From the formula for a geometric series
\[
\sum_{k=0}^{N} e(x k)  	= 		\frac{e(x (N+1))-1}{e(x)-1} \, .
\]
It is possible to pull convenient factors out of the numerator and denominator so that what is left is purely imaginary
\[
\frac{e(x (N+1)) - 1}{e(x) - 1} = \frac{e(x (N+1)/2)}{e(x /2)}\frac{e(x (N+1)/2)-e(-x (N+1)/2)}{e(x /2)-e(-x/2)} \, .
\]
We substitute $y =x (N+1)/2$ into the formula $\sin (2\pi y) = (e(y) - e(-y))/(2i)$ and get
\[
\sum_{k=0}^{N} e(x k) =e(x N/2)\frac{\sin(\pi x (N+1))}{\sin(\pi x)} \, ,
\]
as required.
\end{proof}

\begin{lemma}\label{lemma:1/sin}
Let $k \in \N$.
Let $x \in \R$ such that $x$ is not an integer multiple of $k$.
Then
\[
	\frac{1}{\left|\sin(\frac{\pi}{k}x)\right|} \leq  \frac{k}{2|x|_k} \, .
\]
\end{lemma}

\begin{proof}
Applying the elementary bound\footnote{The easiest way to see why this is true is by drawing the graphs of both functions.} $|\sin(\pi x)| \geq  2\| x\|$ for all $x\in \R$ we get $1/|\sin(\pi x /k)| \leq 1/(2\| x /k\|) = 1/(2|x|_k/k)$ as required.
\end{proof}

\section{Functions with a single coefficient}

We start our analysis by considering functions $f$ with only one non-zero coefficient, i.e. $f(x) = \widehat{f}(\alpha)\chi_\alpha(x)$. 
In Lemma~\ref{lemma:bound_character} we give an upper bound on the size of each Fourier coefficient of $\widetilde{f}$ and a lower bound on the size of the $\lfloor \frac{m}{n}\alpha \rceil$-th Fourier coefficient of $\widetilde{f}$.

\begin{lemma} \label{lemma:bound_character}
Let $m, n\in \N$.
Let $f:\Z_n\to \C$ be a $1$-sparse function, i.e. $f(x) = \widehat{f}(\alpha)\chi_\alpha(x)$ for some $\alpha \in \Z_n$.
Let $\ell := \min(n,m)$.
Let $g:\Z_m\to \C$ be defined by $g(x) = f(x) $ for $0 \leq x < \ell$ and $g(x) = 0$ otherwise.
If $\beta = \frac{m}{n}\alpha$ then $|\widehat{g}(\beta)| = \frac{\ell}{m}|\widehat{f}(\alpha)|$. For all $\beta\in \Z_m$ with $\beta \neq \frac{m}{n}\alpha$ we have
\begin{equation} \label{eq:upperbound}
		|\widehat{g}(\beta)|
	\leq |\widehat{f}(\alpha)| \min \left( \frac{\ell}{m}, \frac{1}{2 |\frac{m}{n}\alpha - \beta|_m}\right) .
\end{equation}
Moreover if $\beta = \lfloor \frac{m}{n}\alpha \rceil$ then
\begin{equation} \label{eq:lowerbound}
		|\widehat{g}(\beta)|
	\geq |\widehat{f}(\alpha)|	\max\left( 		\frac{2 \ell}{\pi m},  \enskip 	\frac{\ell}{m}\sqrt{1 - \frac{\pi^2 \ell^2}{3m^2}\left|  \frac{m}{n}\alpha  - \beta \right |_m^2} \right) .
\end{equation}
\end{lemma}

\begin{proof}
From Lemma \ref{lemma:ghat}
\[
|\widehat{g}(\beta)|
= |\widehat{f}(\alpha)|\frac{1}{m}\left| \sum_{x \in \Z_\ell}\omega_{m}^{(\frac{m}{n}\alpha-\beta) x}\right|
= |\widehat{f}(\alpha)|\frac{1}{m}\left| \sum_{x=0}^{\ell-1}\exp\left(\frac{2\pi i x}{m} \left(\frac{m}{n}\alpha-\beta\right)\right)\right|
\leq |\widehat{f}(\alpha)|\frac{\ell}{m}\,.
\]

Since $\alpha\in\Z_n$ and $\beta\in\Z_m$ the expression $(\frac{m}{n}\alpha-\beta)/m$ is an integer if and only if $\beta = \frac{m}{n}\alpha$.
In this case we get that $|\widehat{g}(\beta)| = |\widehat{f}(\alpha)|\frac{\ell}{m}$.
Otherwise, from Lemma \ref{lemma:weyl}
\begin{equation} \label{eq:ghatbeta}
|\widehat{g}(\beta)|
= |\widehat{f}(\alpha)| \frac{1}{m} \left|\frac{\sin\left(\frac{\pi}{m} \left(\frac{m}{n}\alpha-\beta \right)\ell\right)} {\sin\left(\frac{\pi}{m} \left(\frac{m}{n}\alpha-\beta \right)\right)}\right| .
\end{equation}
We bound the right-hand side of this equation using

\begin{equation} \label{eq:ghatbound}
\frac{1}{m} \left|\frac{\sin\left(\frac{\pi}{m} \left(\frac{m}{n}\alpha-\beta \right)\ell\right)} {\sin\left(\frac{\pi}{m} \left(\frac{m}{n}\alpha-\beta \right)\right)}\right|
\leq \frac{1}{m} \frac{1} {\left|\sin\left(\frac{\pi}{m}\left(\frac{m}{n}\alpha - \beta\right)\right)\right|}
\leq \frac{1}{2 |\frac{m}{n}\alpha - \beta|_m} \, ,
 \end{equation}
where the last inequality follows from substituting $x=\frac{m}{n}\alpha - \beta$ into Lemma \ref{lemma:1/sin}. This proves~(\ref{eq:upperbound}).

For the second part let $r = \left|\frac{m}{n}\alpha - \beta \right|_m$ where $\beta = \lfloor\frac{m}{n}\alpha\rceil$.
Then $0 < r \leq 1/2$. We now substitute our new variable, $r$, into the right-hand side of Equation (\ref{eq:ghatbeta}) and bound the resulting expression with the elementary inequalities $\frac{2}{\pi}|x| \leq |\sin x|$ for all  $x\in [-\pi /2, \pi /2]$ and $ |\sin x|\leq |x|$ for all $x\in \R$.
We get
\[
\frac{1}{m} \left|\frac{\sin\left(\frac{\pi}{m} \left(\frac{m}{n}\alpha-\beta \right)\ell\right)} {\sin\left(\frac{\pi}{m} \left(\frac{m}{n}\alpha-\beta \right)\right)}\right|
= \frac{1}{m}\Bigg|\frac{\sin\left(\frac{\pi \ell}{m} r\right)}{\sin\left(\frac{\pi}{m}r\right)}\Bigg|
\geq \frac{2 \ell}{\pi m} \, .
\]

To get a lower bound that depends on $r$ we first observe that, since $\sin^2(x) = (1-\cos 2x )/2$ holds for all $x\in \R$, the Maclaurin series of $\sin^2(x)$ is $x^2 - x^4/3+  \text{high order terms} $.
Hence $x^2 - x^4/3 \leq \sin^2(x)$ for all $x\in \R$.
Combining this with the elementary bound $\sin^2(x) \leq x^2$ we get
\[
\frac{1}{m^2}\Bigg|\frac{\sin\left(\frac{\pi \ell}{m} r\right)}{\sin\left(\frac{\pi}{m}r\right)}\Bigg|^2
\geq \frac{1}{m^2} \frac{\left(\frac{\pi \ell}{m} r\right)^2 - \frac{\left(\frac{\pi \ell}{m} r\right)^4}{3}}{\left(\frac{\pi}{m}r\right)^2}
=\frac{\ell^2}{m^2}\left(1 - \frac{\pi^2 \ell^2}{3m^2}r^2\right) .
\]
So $|\widehat{g}(\beta)|\geq |\widehat{f}(\alpha)|\frac{\ell}{m}\sqrt{1 - \frac{\pi^2 \ell^2}{3m^2}r^2}$ as required. This proves Equation (\ref{eq:lowerbound}) and completes the proof.
\end{proof}

\begin{remark}
One natural question is whether we can get any nontrivial lower bound on $|\widehat{g}(\beta)|$ where $\beta$ is not the closest integer to $\frac{m}{n}\alpha$.	
First notice that in the proof we rely on $r = |\frac{m}{n}\alpha - \beta|_m$ being at most $1/2$ in order to use the bound $|\sin x|\geq \frac{2}{\pi}|x|$. Since this holds for at most two $\beta$, we can not use the same idea to lower bound other coefficients of $g$.

We first notice that if $m$ and $n$ are coprime and $\beta \neq \frac{m}{n}\alpha$ then $\|r\|\geq\frac{1}{n}$. Using $2\| x\| \leq \sin(\pi x) \leq \pi \| x\|$ for all $x\in\R$ we get
\[
\frac{1}{m} \left|\frac{\sin\left(\frac{\pi}{m} r\ell\right)} {\sin\left(\frac{\pi}{m} r\right)}\right|
\geq \frac{2}{\pi }\frac{\left \| \ell r /m \right \|}{m\|r/m\|}
= \frac{2}{\pi}\frac{\left \|\ell r /m \right \|}{|r|_m}
\geq
\begin{cases}
2/(\pi m |r|_m), 	&\text{if } n<m, \\
2/(\pi n |r|_m), 	&\text{if } m<n.
\end{cases}
\]

It turns out we can not do much better than this bound, as
\[
\frac{1}{m} \left|\frac{\sin\left(\frac{\pi}{m} r\ell\right)} {\sin\left(\frac{\pi}{m} r\right)}\right|
\leq \frac{\pi}{2}\frac{\left \| \ell r /m \right \|}{m\|r/m\|}
= \frac{\pi}{2}\frac{\left \|\ell r /m \right \|}{|r|_m} \, .
\]
Recall $r = |\frac{m}{n}\alpha - \beta|_m$, and if we take $m=n+1$, then we get $r = |\alpha - \beta + \frac{1}{n}\alpha|_m$.
If we take $\beta = \lfloor\frac{m}{n}\alpha\rceil-1 = \alpha-1$ and $\alpha=2$, then $|\widehat{f}(\alpha)|/(\pi m)\leq|\widehat{g}(\beta)| \leq |\widehat{f}(\alpha)|\pi/m$.
\end{remark}

\section{Proofs of the Main Theorems}

In this section we prove a sequence of propositions from which Theorem~\ref{thm:concentrated} and Theorem~\ref{thm:e-concentrated} will follow.
To illustrate the ideas we first consider a single character $\chi_\alpha$, and show that $\widetilde{\chi}_\alpha$ has small number of large coefficients.
Secondly, we show the same result for a $k$-sparse function.
Lastly, we consider any function with $k$ large coefficients and give the same result.

\begin{proposition}\label{prop:1-sparse}
Let $m, n\in \N$.
Let $\epsilon > 0$.
Let $f:\Z_n\to \C$ be a character, i.e. $f(x) = \chi_\alpha(x)$ for some $\alpha \in \Z_n$.
Let $\ell := \min(n,m)$.
Let $g:\Z_m\to \C$ be defined by $g(x) = f(x) $ for $0 \leq x < \ell$ and $g(x) = 0$ otherwise.
Let $\Gamma '$ be the interval in $\Z_m$ defined by $\Gamma' = \{\beta \in \Z_m \mid |\frac{m}{n}\alpha - \beta |_m \leq  r+1\}$ where $r = 1/(2\epsilon)$.
Then
\[
	\|g-g|_{\Gamma'}\|_2^2 	\leq 	\epsilon \, .
\]
\end{proposition}

\begin{proof}
We begin by using the bound from Lemma \ref{lemma:bound_character}
\[
	\|g-g|_{\Gamma'}\|_2^2
	= \sum_{\beta \not\in \Gamma '}|\widehat{g}(\beta)|^2
	\leq \sum_{\beta \not\in \Gamma '}\frac{1}{4|\frac{m}{n}\alpha - \beta|_m^2}\, .
\]
The set $\Gamma '$ consists of all elements of $\Z_m$ whose distance is at most $r+1$ from $\frac{m}{n}\alpha $.
Thus the set $\Z_m\setminus \Gamma'$ can be written as $L\cup R$ where $L$ consists of all elements ``to the left" of $\frac{m}{n}\alpha$, i.e.
$L = \left\{\lceil\frac{m}{n}\alpha  \rceil-x \mid  x = r+1, r+2, \dots , \lceil m/2 \rceil  \right\}\subseteq \Z_m$ and $R$ is all elements ``to the right" of $\frac{m}{n}\alpha$, i.e. $R = \left\{\lfloor \frac{m}{n}\alpha  \rfloor +x \mid  x = r+1, r+2, \dots , \lceil m/2 \rceil  \right\}\subseteq \Z_m$. Then our sum becomes
\[
	\|g-g|_{\Gamma'}\|_2^2
	\leq \sum_{\beta \in L\cup R}\frac{1}{4|\frac{m}{n}\alpha - \beta|_m^2}
	\leq 2\sum_{x= r+1}^{\lceil m/2 \rceil}\frac{1}{4x^2}
	\leq \frac{1}{2r} \\
	= \epsilon \, .
\]
where in the last inequality we use the fact that $\sum_{x = r+1}^{\infty}\frac{1}{x^2} \leq \frac{1}{r}$.
\end{proof}

\begin{proposition}\label{prop:Gamma-sparse}
Let $m, n\in \N$.
Let $\epsilon > 0$.
Let $f:\Z_n\to \C$ be a $|\Gamma |$-sparse function, i.e. $f(x) = \sum_{\alpha \in \Gamma}\widehat{f}(\alpha)\chi_\alpha(x)$.
Let $\ell := \min(n,m)$.
Let $g:\Z_m\to \C$ be defined by $g(x) = f(x) $ for $0 \leq x < \ell$ and $g(x) = 0$ otherwise.
Let $\Gamma ' =  \bigcup_{\alpha\in \Gamma}\{\beta \in \Z_m \mid |\frac{m}{n}\alpha - \beta |_m \leq r+1\}$ for $r = |\Gamma|\|f\|_2^2 / (2\epsilon)$.
Then
\[
\|g - g|_{\Gamma '} \|_2^2 	\leq 	\epsilon \, .
\]
\end{proposition}

\begin{proof}
We can bound $|\widehat{g}(\beta)|$ with the triangle inequality
\[
|\widehat{g}(\beta)|
= |\langle g, \chi_{\beta, m} \rangle |
= \Bigg| \langle \sum_{\alpha \in \Gamma} \widehat{f}(\alpha) \widetilde{\chi}_{\alpha, n}, \chi_{\beta, m}\rangle \Bigg|
 \leq \sum_{\alpha \in \Gamma} |\widehat{f}(\alpha)||\langle \widetilde{\chi}_{\alpha, n} , \chi_{\beta, m} \rangle | \, .
\]
Then
\[
\| g - g|_{\Gamma '} \|_2^2 =\sum_{\beta \not\in \Gamma '}|\widehat{g}(\beta)|^2
\leq \sum_{\beta \not\in \Gamma '}\Bigg|\sum_{\alpha \in \Gamma} |\widehat{f}(\alpha)||\langle \widetilde{\chi}_{\alpha, n} , \chi_{\beta, m} \rangle | \Bigg|^2 \, .
\]
The Cauchy--Schwarz inequality gives
\begin{equation*}
\begin{split}
\|g-g|_{\Gamma '}\|_2^2 &\leq \sum_{\beta \not\in \Gamma '}\Bigg(\sum_{\alpha \in \Gamma} |\widehat{f}(\alpha)|^2 \sum_{\alpha \in \Gamma}|\langle \widetilde{\chi}_{\alpha, n} , \chi_{\beta, m} \rangle |^2 	\Bigg)  \\
&= \sum_{\beta \not\in \Gamma '} \Bigg( \|f\|_2^2 \sum_{\alpha \in \Gamma }|\langle \widetilde{\chi}_{\alpha, n} , \chi_{\beta, m} \rangle |^2\Bigg) \\
&= \|f\|_2^2\sum_{\alpha \in \Gamma }\sum_{\beta \not\in \Gamma '}|\langle \widetilde{\chi}_{\alpha, n} , \chi_{\beta, m} \rangle |^2 \, .
\end{split}
\end{equation*}

For the last part of the proof we bound $\sum_{\beta \not\in \Gamma '}|\langle \widetilde{\chi}_{\alpha, n} , \chi_{\beta, m} \rangle |^2$ using Proposition~\ref{prop:1-sparse} with $\epsilon $ replaced by $\frac{\epsilon}{|\Gamma |\|f\|_2^2}$ \big(as $\langle \widetilde{\chi}_{\alpha, n} , \chi_{\beta, m} \rangle = \widehat{\widetilde{\chi}}_{\alpha, n}(\beta)\big)$.
This gives
\[
	\|g-g|_{\Gamma '}\|_2^2
	\leq  \|f\|_2^2\sum_{\alpha \in \Gamma }\frac{\epsilon}{|\Gamma |\|f\|_2^2}
	= \epsilon \, .
\]
\end{proof}

\begin{proposition}\label{prop:Gamma-concentrated}
Let $m, n\in \N$.
Let $f:\Z_n \to \C$.
Suppose $\Gamma \subseteq \Z_n$ is such that
\[
	\|f-f|_\Gamma \|_2^2 		\leq  	\epsilon
\]
for some $\epsilon > 0$.
Let $\ell := \min(n,m)$.
Let $g:\Z_m\to \C$ be defined by $g(x) = f(x) $ for $0 \leq x < \ell$ and $g(x) = 0$ otherwise.
Let $\epsilon'>0$.
Let $\Gamma ' =  \bigcup_{\alpha\in \Gamma}\{\beta \in \Z_m \mid |\frac{m}{n}\alpha - \beta |_m \leq r+1\}$ for $r =|\Gamma|\|f\|_2^2 / (2\epsilon')$.
Then 
\[
\|g-g|_{\Gamma '}\|_2^2
\leq  t \epsilon + \epsilon' + 2\sqrt{t\epsilon\epsilon'} \, ,
\]
where $t = n/m$.
\end{proposition}

\begin{proof}
We write $f = f|_{\Gamma} + f|_{\Z_n\setminus \Gamma}$.
From linearity of the $ \widetilde{\cdot} $ operator
\[
g = \widetilde{f} = \widetilde{f}|_{\Gamma} + \widetilde{f}|_{\Z_n\setminus \Gamma} \, .
\]	
There is a possibility for some confusion with the notation. By $\widetilde{f}|_{\Gamma}$ we mean $\widetilde{(f|_\Gamma)}$, not $\widetilde{f}$ projected on $\Gamma$ (the latter is not defined as $\Gamma \subseteq \Z_n$).
By linearity of the Fourier Transform
\[
	\widehat{g}(\beta) = \widehat{\widetilde{f_{\;}}}|_\Gamma(\beta) + \widehat{\widetilde{f_{\;}}}|_{\Z_n\setminus \Gamma}(\beta)
\]
for all $\beta \in \Z_m$. Hence
\begin{equation*}
\begin{split}
\| g - g|_{\Gamma '}\|_2^2
&= \sum_{\beta \not\in \Gamma '}|\widehat{g}(\beta )|^2
= \sum_{\beta \not\in \Gamma '}\left|\widehat{\widetilde{f_{\;}}}|_\Gamma(\beta ) + \widehat{\widetilde{f_{\;}}}|_{\Z_n \setminus \Gamma}(\beta )\right|^2 \\
&\leq \sum_{\beta \not\in \Gamma '}\big|\widehat{\widetilde{f_{\;}}}|_\Gamma(\beta ) \big|^2 + 2\sum_{\beta \not\in \Gamma '}\big|\widehat{\widetilde{f_{\;}}}|_\Gamma(\beta ) \big| \big|\widehat{\widetilde{f_{\;}}}|_{\Z_n \setminus \Gamma }(\beta )\big| + \sum_{\beta \not\in \Gamma '}\big|\widehat{\widetilde{f_{\;}}}|_{\Z_n \setminus \Gamma}(\beta )\big|^2  .
\end{split}
\end{equation*}

The rest of the proof consists of bounding the three terms on the right-hand side of the above equation.
Since $f|_\Gamma :\Z_n\to \C $ is $|\Gamma |$-sparse we can apply Proposition \ref{prop:Gamma-sparse} with $\epsilon'$, which says
\[
	\sum_{\beta \not\in \Gamma '}\big|\widehat{\widetilde{f_{\;}}}|_\Gamma(\beta) \big|^2 = \|\widetilde{f}|_\Gamma-\big(\widetilde{f}|_\Gamma\big)|_{\Gamma'} \|_2^2 \leq \epsilon'  .
\]
By hypothesis $\|f|_{\Z_n\setminus \Gamma }\|_2^2 = \|f-f|_\Gamma \|_2^2 \leq \epsilon$. Hence
\[
	\sum_{\beta \not\in \Gamma '}\big|\widehat{\widetilde{f_{\;}}}|_{\Z_n \setminus \Gamma}(\beta ) \big|^2
	\leq \sum_{\beta \in \Z_n}\big|\widehat{\widetilde{f_{\;}}}|_{\Z_n \setminus \Gamma}(\beta )\big|^2
	= \|\widetilde{f}|_{\Z_n \setminus \Gamma}\|_2^2
	\leq \frac{n}{m}\|f|_{\Z_n\setminus \Gamma }\|_2^2
	\leq t \epsilon \, .
\]
	
Finally by the Cauchy--Schwarz inequality
\[2\sum_{\beta \not\in \Gamma '}\big|\widehat{\widetilde{f_{\;}}}|_\Gamma(\beta ) \big| \big|\widehat{\widetilde{f_{\;}}}|_{\Z_n \setminus \Gamma }(\beta )\big|
\leq 2 \sqrt{\sum_{\beta \not\in \Gamma '}\big|\widehat{\widetilde{f_{\;}}}|_\Gamma(\beta ) \big|^2 \sum_{\beta \not \in \Gamma'} \big|\widehat{\widetilde{f_{\;}}}|_{\Z_n \setminus \Gamma}(\beta ) \big|^2 }
\leq 2\sqrt{\epsilon' t \epsilon} \, .\]

Putting it all together we get
\begin{equation*}
\begin{split}
\sum_{\beta \not\in \Gamma '}|\widehat{\widetilde{g}}(\beta )|^2
&\leq \sum_{\beta \not\in \Gamma '}\big|\widehat{\widetilde{f_{\;}}}|_\Gamma(\beta ) \big|^2 +  2\sum_{\beta \not\in \Gamma '}\big|\widehat{\widetilde{f_{\;}}}|_\Gamma(\beta ) \big| \big|\widehat{\widetilde{f_{\;}}}|_{\Z_n \setminus \Gamma }(\beta )\big| + \sum_{\beta \not\in \Gamma '}\big|\widehat{\widetilde{f_{\;}}}|_{\Z_n \setminus \Gamma}(\beta )\big|^2 \\
&\leq \epsilon' + 2\sqrt{\epsilon' t \epsilon} + t \epsilon \, .
\end{split}
\end{equation*}
\end{proof}

\begin{proof}[\textbf{Proof of Theorem \ref{thm:concentrated}}]
Let $\epsilon >0$.
By hypothesis the family $\{f_k:\Z_{n_k}\to\C\}_{k\in\N}$ is concentrated. Hence there exists a bivariate polynomial $P\in \R[x, y]$ and sets $\Gamma_k\subseteq \Z_{n_k}$ such that $|\Gamma_k|\leq P(\log(n_k),\frac{1}{\epsilon / 6})$ and $\|f_k - f_k|_{\Gamma_k}\|_2^2 \leq \epsilon/6$ for all $k\in \N$.

By Proposition \ref{prop:Gamma-concentrated} with $\epsilon ' = \epsilon /6$ there exist sets $\Gamma_k'\subseteq \Z_{m_k}$ such that
$
	\|\widetilde{f}_k - \widetilde{f}_k|_{\Gamma_k'}\|_2^2
	\leq t_k\epsilon /6 + \epsilon /6 + 2\sqrt{t_k}(\epsilon/6) ,
$
for all $k\in \N$, where $t_k := n_k/m_k$.
By hypothesis $n_k/m_k\leq 2$ so
\[
	\|\widetilde{f}_k - \widetilde{f}_k|_{\Gamma_k'}\|_2^2
	\leq 2 \epsilon /6 + \epsilon /6 + 2\sqrt{2}(\epsilon/6)< \epsilon
\]
for all $k\in \N$.
All that remains to show is that the sets $\Gamma_k'$ have polynomial size.
Since Proposition \ref{prop:Gamma-concentrated} gives an explicit construction we know that each $\Gamma_k'$ is a union of $|\Gamma_k|$ intervals of width at most $2(r+1)+1=2r+3$ where $r= |\Gamma_k|\|f_k\|_2^2/(2\epsilon')$ hence
\[
|\Gamma_k'|
\leq |\Gamma_k|(2r+3)
= |\Gamma_k|\left(2\frac{|\Gamma_k |  \|f_k\|_2^2}{2\epsilon'}+3\right)
= |\Gamma_k|\left(\frac{6|\Gamma_k |  \|f_k\|_2^2}{\epsilon}+3\right)
\]

By hypothesis there exists a polynomial $Q\in \R[x]$ such that $\|f_k\|_2^2\leq Q(\log(n_k))$ for all $k\in \N$. When we upper bound $|\Gamma_k|$ with $P(\log(n_k),6/\epsilon)$ and upper bound $\|f_k\|_2^2$ with $Q(\log(n_k))$ the displayed inequality becomes
\[
	|\Gamma_k'|
	\leq P(\log(n_k),6/\epsilon)\left(6\epsilon^{-1}P(\log(n_k),6/\epsilon) Q(\log(n_k)) + 3\right).
\]
By hypothesis $n_k/2\leq m_k$, and so $\log(n_k)\leq\log(m_k)+1$. Therefore, the expression is polynomial in $1/\epsilon$ and $\log(m_k)$ as required. So the definition of concentration is satisfied with the polynomial $P'(x, y) =  P(x+1, 6y)(6yP(x+1, 6y)Q(x+1) + 3)$.
\end{proof}

\begin{proof}[\textbf{Proof of Theorem \ref{thm:e-concentrated}}]
The proof of this theorem is similar to that of Theorem \ref{thm:concentrated}.
By hypothesis the family $\{f_k:\Z_{n_k}\to\C\}_{k\in\N}$ is $\epsilon$-concentrated. Hence there exists a polynomial $P\in \R[x]$ and there exist sets $\Gamma_k\subseteq \Z_{n_k}$ such that for all $k\in \N$ we have $|\Gamma_k|\leq P(\log(n_k))$ and $\|f_k - f_k|_{\Gamma_k}\|_2^2 \leq \epsilon$.

From Proposition \ref{prop:Gamma-concentrated} there exist sets $\Gamma_k'\subseteq \Z_{m_k}$ such that $\|\widetilde{f}_k - \widetilde{f}_k|_{\Gamma_k'}\|_2^2 \leq t\epsilon + \epsilon' + 2\sqrt{t\epsilon\epsilon'}$ for all $k\in \N$ (recall that $t:= \sup_{k\in \N}\{n_k/m_k\}$).
If we choose $\epsilon' = \min(\eta / 2, \eta^2/(16t\epsilon))$ then we get
\[
\|\widetilde{f}_k - \widetilde{f}_k|_{\Gamma_k'}\|_2^2
\leq t\epsilon + \eta/2 + 2\sqrt{t\epsilon\frac{\eta^2}{16t\epsilon}}
= t\epsilon + \eta/2 + \eta/2
= t\epsilon + \eta \, .
\]
Moreover, Proposition \ref{prop:Gamma-concentrated} gives us a bound for $|\Gamma_k'|$. Recall that $\Gamma_k'$ is a union of $|\Gamma_k|$ intervals of at most width $2r+3$ where $r= |\Gamma_k|\|f_k\|_2^2/(2\epsilon')$ hence
\[
|\Gamma_k'|
\leq |\Gamma_k|(2r+3)
= |\Gamma_k|\left(2\frac{|\Gamma_k | \|f_k\|_2^2}{2\epsilon'}+3\right) .
\]
By hypothesis there exists a polynomial $Q\in \R[x]$ such that $\|f_k\|_2^2\leq Q(\log(n_k))$ for all $k\in \N$. When we upper bound $|\Gamma_k|$ with $P(\log(n_k))$ and upper bound $\|f_k\|_2^2$ with $Q(\log(n_k))$ the displayed inequality becomes
\[
|\Gamma_k'|
\leq P(\log(n_k)) \left(P(\log(n_k))Q(\log(n_k))/\epsilon' + 3\right) .
\]
By hypothesis $n_k/2\leq m_k$, and so $\log(n_k)\leq\log(m_k)+1$. Since $\epsilon'$ does not depend on $k$ this expression is polynomial in $\log(m_k)$ as required.\footnote{However, since $\epsilon'$ depends on $\epsilon,\eta$ one can redefine a global polynomial $P'(x,y)$ in those terms to show the affect of $\eta$ on the polynomial. This is not needed for our proof, as the only important variable is the domain size.} So the family $\{\widetilde{f}_k:\Z_{m_k}\to \C\}_{k\in \N}$ is $(t\epsilon + \eta)$-concentrated with respect to the polynomial $P'(x) = P(x+1)(P(x+1)Q(x+1)/\epsilon' + 3)$.
\end{proof}

\section{Coefficient Cancellation} \label{sec:coeff-cancel}

In this section we give some examples of functions where $|\widehat{f}(\alpha)|$ is large but $|\widehat{\widetilde{f_{\;}}}(\lfloor \frac{m}{n}\alpha \rceil)|$ is not.
We then show that $|\widehat{f}(\alpha)|$ being large implies $|\widehat{\widetilde{f_{\;}}}(\lfloor \frac{m}{n}\alpha \rceil)|$ is large if $f$ satisfies certain conditions.

Given some function $f:\Z_n\to \C$, which can be well approximated by $f|_\Gamma$, Proposition \ref{prop:Gamma-concentrated} shows us how to construct a set $\Gamma '\subseteq \Z_m$ such that $\widetilde{f}:\Z_m\to \C$ can be well approximated by $\widetilde{f}|_{\Gamma'}$.

However, the theorem does not guarantee that $|\widehat{\widetilde{f_{\;}}}(\beta)|$ is large for every $\beta \in \Gamma'$. So, if we choose $\Gamma'$ as in the theorem, this will be sufficient to get a good approximation but some of the members of $\Gamma'$ may not be necessary for that approximation.

In the case where the function $f$ is a single character we proved a stronger result: if $f = \chi_\alpha$ then $\widetilde{f}=\widetilde{\chi}_\alpha$ and $\widehat{\widetilde{f_{\;}}}(\lfloor \frac{m}{n}\alpha \rceil )$ is greater than $2\ell/(\pi m)$ where $\ell = \min(n,m)$. Thus if we want to closely approximate $\widetilde{f}$ it is not only sufficient but necessary for $\Gamma '$ to contain $\lfloor \frac{m}{n}\alpha \rceil$.
Since $\widetilde{f} = \sum_{\alpha \in \Z_n}\widehat{f}(\alpha)\widetilde{\chi}_{\alpha}$ is just a sum of $\widetilde{\chi}_\alpha$ one might expect this result can be generalised to arbitrary functions. However there is no easy way to conclude which coefficients of $\widetilde{f}$ are large, or even how many of them there are.
Observe that
\[
\widehat{\widetilde{f_{\;}}}(\beta)
= \sum_{\alpha \in \Z_n} \widehat{f}(\alpha) \langle \widetilde{\chi}_{\alpha, n} , \chi_{\beta, m} \rangle
= \widehat{f}(\alpha') \langle \widetilde{\chi}_{\alpha', n} , \chi_{\beta, m} \rangle + \sum_{\alpha' \neq \alpha \in \Z_n} \widehat{f}(\alpha) \langle \widetilde{\chi}_{\alpha, n} , \chi_{\beta, m} \rangle\, .
\]
So even if $\widehat{f}(\alpha') \langle \widetilde{\chi}_{\alpha', n} , \chi_{\beta, m} \rangle$ is large, e.g. when $\alpha' = \lfloor \frac{m}{n}\alpha \rceil $, it may be cancelled out by the other term.

\begin{example}
Let $n$ be an odd number and consider the function $f:\Z_n\to \C$ given by
\[
 f(x) :=
	 \begin{cases}
		 1 	& \text{ when $x$ is even} , \\
		-1 	& \text{ when $x$ is odd}.
	\end{cases}
\]
This function has several large coefficients centred at $n/2$ (follows from Lemma \ref{lemma:weyl}; see also Eq. (\ref{eq:bits}) in Lemma \ref{lemma:bits} below).
Let $m=n + 1$ and let $d = m/2\in \N$.
We show that $d$-th Fourier coefficient of $\widetilde{f}:\Z_m\to \C$ is large and that the magnitude of the rest is $1/m$.
Notice that for every $0\leq x<n$ we have $\widetilde{f}(x) = \chi_{d,m}(x)$.
Therefore
\begin{equation*}
\begin{split}
\widehat{\widetilde{f_{\;}}}(\beta)
&= \frac{1}{m}\sum_{x\in\Z_m}\widetilde{f}(x)\overline{\chi_\beta(x)}
= \frac{1}{m}\sum_{0\leq x<n}\widetilde{f}(x)\overline{\chi_\beta(x)}
= \frac{1}{m}\sum_{0\leq x<n}\chi_{d-\beta,m}(x) \\
&= \frac{1}{m}\sum_{0\leq x<n}\chi_{d-\beta,m}(x) + \frac{1}{m}\sum_{n\leq x<m}\chi_{d-\beta,m}(x) - \frac{1}{m}\sum_{n\leq x<m}\chi_{d-\beta,m}(x) \\
&= \frac{1}{m}\sum_{0\leq x<m}\chi_{d-\beta,m}(x) - \frac{1}{m}\sum_{n\leq x<m}\chi_{d-\beta,m}(x) \, .
\end{split}
\end{equation*}
Therefore, for $\beta\neq d$, we get $|\widehat{\widetilde{f_{\;}}}(\beta)| = |\frac{1}{m}\sum_{0\leq x<m}\chi_{d-\beta,m}(x) - \frac{1}{m}\chi_{d-\beta,m}(n)| = 0 + \frac{1}{m}$.
Since $\|\widetilde{f}\|_2^2 = \frac{1}{m}(m-1)$ from Parseval's identity we get $|\widehat{\widetilde{f_{\;}}}(d)| > \frac{1}{m}(m-1) - (m-1)\frac{1}{m^2} = 1 - O(1/m)$. We therefore see that all large coefficients of $f$, except for one, are canceled out in $\widetilde{f}$. This example also works for $m=n+1+2k$ for some $k$ not too large, e.g. $k\approx\sqrt{n}$.
\end{example}

\begin{example}\label{ex:switch_down}
Let $n$ be even. Let $m = n/2$. Define
\[
	f(x) :=
	\begin{cases}
		0					& \text{ when } 0 \le x < m , \\
		\chi_{\alpha, n}(x) & \text{ when } m \le x < n.
	\end{cases}
\]
Then $\widehat{f}(\alpha)$ is large. However $\widetilde{f}$ is the zero function, hence all its Fourier coefficients are zero. In the appendix we refine this example to show that the bound in Theorem~\ref{thm:e-concentrated} is tight.
\end{example}

Recall from Lemma \ref{lemma:ghat} that $\widehat{\widetilde{f_{\;}}}(\beta)$ is a weighted sum of the Fourier coefficients of $f$,
\[
	\widehat{\widetilde{f_{\;}}}(\beta)
	= \sum_{\alpha \in \Z_n} \widehat{f}(\alpha) 	\frac{1}{m} 		\sum_{x \in \Z_\ell} \omega_{m}^{(\frac{m}{n}\alpha-\beta) x} \,.
\]
In addition, the weight of the $\alpha$-th Fourier coefficient, $\frac{1}{m} \sum_{x \in \Z_\ell} \omega_{m}^{(\frac{m}{n}\alpha-\beta) x}$, is bounded by $1/(2|\frac{m}{n}\alpha-\beta|_m)$ (See Equations (\ref{eq:ghatbeta}) and (\ref{eq:ghatbound}) in the proof of Lemma \ref{lemma:bound_character}).
Therefore if $f$ has a few large Fourier coefficients and they are all spaced far apart, then the only term in the sum that will contribute significantly to $\widehat{\widetilde{f_{\;}}}(\beta)$ is the $\alpha$-th term where $|\widehat{f}(\alpha)|$ is large  and $|\frac{m}{n}\alpha-\beta|_m$ is minimal.
Proposition \ref{prop:concentrated_apart} formalises this argument.


\begin{proposition}\label{prop:concentrated_apart}
	Let $m, n\in \N$ with $n \leq m \leq 2n$.
	Let $L> 0$.
	Let $C>1$.
	Let $\Gamma \subseteq \Z_n$.
	Let $r = 20C(\log (|\Gamma|/2)+1)$.
	Let $\tau <  \frac{L}{20}\frac{1}{3+2\log(n)}$.
	Let $f:\Z_n\to \C$,
    and suppose $\Gamma\subseteq \Z_n$ satisfies $|\alpha - \alpha'|_n > r$ for all $\alpha, \alpha'\in \Gamma$ with $\alpha\neq \alpha'$, and that
	\[
		L \leq |\widehat{f_{\;}}(\alpha)|\leq CL
	\]
	for all $\alpha\in \Gamma$ and
	\[
		|\widehat{f_{\;}}(\alpha)|\leq \tau
	\]
	for all $\alpha\not\in \Gamma$.
	Then
	\[
		|\widehat{\widetilde{f_{\;}}}(\lfloor \tfrac{m}{n}\alpha \rceil )|\geq \frac{1}{5}L \,
	\]
	for all $\alpha\in \Gamma$.
\end{proposition}

\begin{proof}[Proof sketch]
Let $\gamma\in\Gamma$ and $\beta = \lfloor \tfrac{m}{n}\gamma \rceil$. Then by the reverse triangle inequality
\begin{equation*}
\begin{split}
\big|\widehat{\widetilde{f_{\;}}}(\beta)\big|
&= \Bigg| \sum_{\alpha\in\Z_n}\widehat{f}(\alpha) \langle \widetilde{\chi}_{\alpha, n} , \chi_{\beta, m} \rangle \Bigg| \\
& \geq \left|\Big|\widehat{f}(\gamma)\langle \widetilde{\chi}_{\gamma, n} , \chi_{\beta, m} \rangle\Big| - \Bigg|\sum_{\gamma \neq \alpha\in\Gamma}\widehat{f}(\alpha) \langle \widetilde{\chi}_{\alpha, n} , \chi_{\beta, m} \rangle\Bigg| - \Bigg|\sum_{\alpha\notin\Gamma}\widehat{f}(\alpha) \langle \widetilde{\chi}_{\alpha, n} , \chi_{\beta, m} \rangle\Bigg|\right|  .
\end{split}
\end{equation*}

By Lemma \ref{lemma:bound_character} the first term satisfies $\Big|\widehat{f}(\gamma)\langle \widetilde{\chi}_{\gamma, n} , \chi_{\beta, m} \rangle\Big| \geq \frac{2}{\pi}\frac{n}{m}L$. By hypothesis $n/m \geq 1/2$ so $\Big|\widehat{f}(\gamma)\langle \widetilde{\chi}_{\gamma, n} , \chi_{\beta, m} \rangle\Big| \geq \frac{1}{\pi}L$. Similar to the proof of Proposition~\ref{prop:1-sparse}, we have for the second term
\begin{equation*}
\Bigg|\sum_{\gamma \neq \alpha\in\Gamma}\widehat{f}(\alpha) \langle \widetilde{\chi}_{\alpha, n} , \chi_{\beta, m} \rangle\Bigg|
\leq CL \sum_{\gamma \neq \alpha\in\Gamma} |\langle \widetilde{\chi}_{\alpha, n} , \chi_{\beta, m} \rangle|
\leq \frac{CL}{2} \sum_{\gamma \neq \alpha\in\Gamma} \frac{1}{|\frac{m}{n}\alpha-\beta|_m} \, .
\end{equation*}
Since $\Gamma$ is $r$-spread we know that $\sum_{\gamma \neq \alpha\in\Gamma} \frac{1}{|\frac{m}{n}\alpha-\beta|_m}\leq 2 ( \frac{1}{r} + \frac{1}{2r} + \frac{1}{3r}\cdots )$ so
\begin{equation*}
\Bigg|\sum_{\gamma \neq \alpha\in\Gamma}\widehat{f}(\alpha) \langle \widetilde{\chi}_{\alpha, n} , \chi_{\beta, m} \rangle\Bigg|
 \leq CL \sum_{x=1}^{|\Gamma|/2} \frac{1}{xr} \leq \frac{CL}{r} (\log(|\Gamma|/2)+1)
 = \frac{1}{20}L \, .
\end{equation*}
where we use the fact that $\sum_{x=1}^N \frac{1}{x} < \log(N)+1$ in the second to last inequality.
Similarly, we can bound the third term by $\frac{1}{20}L$ using $|\widehat{f_{\;}}(\alpha)|\leq \tau$ for all $\alpha\not\in \Gamma$. Then $\big|\widehat{\widetilde{f_{\;}}}(\beta)\big| \geq \frac{1}{\pi}L - \frac{1}{20}L - \frac{1}{20}L \geq \frac{1}{5}L $.
\end{proof}

A fully detailed proof can be found in the master's thesis of the first-named author \cite{Joel-thesis}.

\section{Applications}

We give some applications of our results.
In Section \ref{sec:algorithms} we show that our results can be turned into an algorithm to recover the large coefficients of a function $f:\Z_n\to\C$ using the large coefficients of the corresponding function $\widetilde{f}:\Z_{2^k}\to\C$.
Algorithms for finding large coefficients on the latter domain have been highly optimized and are generally simpler than the algorithms that work over any domain.

In Section \ref{sec:conc_proof} we give a method to prove that a family of functions is concentrated by showing that a subfamily is concentrated.
In particular, we show that the $i$-th bit functions are concentrated.

\subsection{An algorithm for finding large coefficients}\label{sec:algorithms}
We sketch how our results can be used to reduce the problem of finding the large coefficients of $f:\Z_n\to \C$ to the case where $n$ is an integer power of 2.

Consider the following computational problem: given a function $f:\Z_n\to\C$ and some threshold $\tau>0$, recover all coefficients of $f$ that satisfy $|\widehat{f}(\alpha)|^2>\tau$.
In other words, the goal is to produce a set $\Gamma$ such that $\alpha\in\Gamma$ if and only if $|\widehat{f}(\alpha)|^2>\tau$.
One way of finding these coefficients would be to calculate all the Fourier coefficients and then discard those with small magnitude.
Of course, this seems wasteful if there are only a small number of large coefficients.
There are algorithms which solve this problem in time $\mathrm{poly}(\log(n), 1/\tau, \|f\|_{\infty})$ for any domain $\Z_n$, yet the case where $n=2^k$ has received the most attention in the literature.
We refer to \cite{Fourier_survey} for a recent survey on these algorithms and their practical applications. We use the name \emph{sparse Fourier transform} (SFT) to refer to these algorithms.

We now present a new approach to this problem: instead of directly finding the large Fourier coefficients of $f:\Z_n\to \C$ we can define $\widetilde{f}:\Z_{2^k}\to \C$ and use an existing SFT algorithm to find the large coefficients of $\widetilde{f}$.
Our results show that whenever $f$ has large coefficients, also $\widetilde{f}$ has large coefficients.
Moreover, since the proofs are constructive, if we know the set of large coefficients of $\widetilde{f}$ then this restricts where the large coefficients of $f$ can be.

We only present an overview of the algorithm; further details can be find in the appendix. Before we describe the algorithm we mention that one can efficiently estimate the magnitude of the $\alpha$-th coefficient of a function $f$ by computing
\[ \left| \frac{1}{k}\sum_{i=1}^k f(x_i) \overline{\chi_\alpha(x_i)}\right| \, ,\]
for $x_i\in\Z_n$ chosen independently and uniformly at random. This is standard and follows from Chernoff's bound. The size of $k$ depends on the required approximation, and can be fairly small for our needs.

\medskip
\noindent\textbf{Main idea.}
From our analysis above, we know that the large Fourier coefficients $\widetilde{f}$ are those which are close to $\lfloor \frac{m}{n}\alpha \rceil$ where $|\widehat{f}(\alpha)|$ is large.
Therefore, if $|\widehat{\widetilde{f_{\;}}}(\beta)|$ is large we can search in a small interval around $\lfloor \frac{n}{m}\beta \rceil$ to find the large coefficients of $f$.

\medskip
\noindent\textbf{Overview.} Given a function $f:\Z_n\to\C$ and some threshold $\tau>0$, we define $\widetilde{f}$ over some desired domain $\Z_m$ for $m>n$.
We say that a coefficient $\widehat{f}(\alpha)$ is \emph{$\tau$-heavy} if $|\widehat{f}(\alpha)|^2>\tau$.
If $f$ has $\tau$-heavy coefficients then $\widetilde{f}$ has $\tau'$-heavy coefficients for some $\tau'>0$.
One recovers a set $\Gamma'$ of the $\tau'$-heavy coefficients of $\widetilde{f}$ using an SFT.

Suppose $\beta \in \Gamma'$, and recall that $\widehat{\widetilde{f_{\;}}}(\beta) = \sum_{\alpha \in \Z_n} \widehat{f}(\alpha) \langle \widetilde{\chi}_{\alpha, n} , \chi_{\beta, m} \rangle$.
The set of coefficients $\widehat{f}(\alpha)$ that significantly contribute to the magnitude of $\widehat{\widetilde{f_{\;}}}(\beta)$ belongs to an interval of (small) size $r$ around $\widehat{f}(\lfloor \frac{n}{m}\beta \rceil)$.
Hence, by approximating the coefficients in this interval, one can recover all the $\tau$-heavy coefficients (in fact, all the large coefficients) that significantly contribute to the magnitude of $\widehat{\widetilde{f_{\;}}}(\beta)$.
We see that by translating back to $\Z_n$ for each $\tau'$-heavy coefficient of $\widetilde{f}$, it is possible to find large coefficients of $f$.

\medskip
\noindent\textbf{Issues with coefficient cancellation.}
However, we need to keep in mind that some large coefficient of $f$ may have been ``canceled out'' in $\widetilde{f}$ as described in Section \ref{sec:coeff-cancel}.
This can be overcome, for example, by extending the search space a little bit more or by subtracting all the known large coefficients from $f$, and iterating.

Another way to overcome the problem of coefficients being ``cancelled out'' is to use the scaling property of the Fourier transform, namely if $c\in\Z^*_n$ and $h(x)=f(cx)$ then $\widehat{h}(\alpha) = \widehat{f}(c^{-1}\alpha)$. The scaling property of the Fourier transform can be used to first permute all the coefficients of $f$ by defining the function $h$ for some randomly chosen $c\in \Z_n^*$.
With high probability the large coefficients of $h$ are spread apart and so recovering them from the large coefficients of $\widetilde{h}$ is easy using Proposition \ref{prop:concentrated_apart}.
We simply find $\Gamma'$ (the set of all Fourier coefficients of $\widetilde{h}$ that are larger than $\tau'$) and isolate those $\beta$ for which $|\widehat{\widetilde{h_{\;}}}(\beta)|$ is the largest is some small interval (if there are any close $\tau'$-heavy coefficients in $\Gamma'$). Then for every such $\beta$ we test to see which one of $|\widehat{h}(\lfloor \frac{n}{m}\beta \rfloor)|$ and $|\widehat{h}(\lceil \frac{n}{m}\beta \rceil)|$ is large.

Once this is done, we can multiply each coefficient by $c^{-1}$ to recover the $\tau$-heavy coefficients of the original function $f$.

\subsection{Proving concentration}\label{sec:conc_proof}

As a second application we use Theorem \ref{thm:concentrated} to give a simple proof that a family of functions is concentrated by considering a concentrated subfamily.

\begin{theorem}
Consider a family of functions $\famlyF = \{ f_{2^k}:\Z_{2^k} \to \C \}_{k\in\N}$ and define the family $\famlyF' = \{ f_n:\Z_n \to \C \}_{n\in\N}$, where for each $2^{k-1} < n \leq 2^k$ we let $f_n(x) := f_{2^k}(x)$ for every $x\in\Z_n$.
If $\famlyF$ is concentrated then $\famlyF'$ is concentrated.
\end{theorem}

\begin{proof}
Consider the family $\famlyG=\{g_n:\Z_{m_n}\to\C\}_{n\in\N}$ where for each $2^{k-1} < n \leq 2^k$ we let $m_n=2^k$ and define $g_n:=f_{2^k}$.
Since $\famlyG$ and $\famlyF$ contain the same functions, $\famlyG$ is concentrated (with the same polynomial for which $\famlyF$ is concentrated).
Note that for each $n$, we have that $f_n = \widetilde{g}_n$ with $1\leq\frac{n}{m_n}<2$.
From Theorem \ref{thm:concentrated}, since $\{g_n:\Z_{m_n}\to\C\}_{n\in\N}$ is concentrated it follows that $\{ f_n:\Z_n \to \C \}_{n\in\N} = \{ \widetilde{g}_n:\Z_n \to \C \}_{n\in\N}$ is concentrated.
\end{proof}

\begin{remark}
The family $\famlyF$ is taken over powers of $2$ to be consistent with Theorem \ref{thm:concentrated} above, but it could be generalised to other subfamilies.
\end{remark}

We now use this theorem to (re)prove that the (family of) $i$-th bit function is concentrated.
This was first proved in \cite{MR} in a quite complicated work.
Our proof is simple and follows from the following result on the subfamily of functions with domain $\Z_{2^k}$.

\begin{lemma}\label{lemma:bits}
Let $k \in \N$ and $0 \le i < k$.
Define $\bit_i : \Z_{2^k} \to \{-1,1\}$ by
$\bit_i(x) = (-1)^{x_i}$ where $x = \sum_{j=0}^{k-1} x_j 2^j$ and $x_j \in \{ 0,1 \}$.
Let $\alpha \in \Z_{2^k}$.
Then $\widehat{\bit_i}(\alpha) = 0$ unless $\alpha$ is an odd multiple of $2^{k-i-1}$ in which case $| \widehat{\bit_i}(\alpha)| = O( 2^{k-i}/|\alpha|_{2^k} )$.
\end{lemma}

\begin{proof}
Writing $N = 2^k$ we have $\widehat{\bit_i}(\alpha) = \tfrac{1}{N} \sum_{x=0}^{2^k-1} \bit_i(x) \omega_N^{-\alpha x}$.
We write $x = y + 2^i b + 2^{i+1} z$ where $0 \le y < 2^i$, $0 \le z < 2^{k - (i+1)}$ and where $b$ is the $i$-th bit.
Then
\begin{equation*}
\begin{split}
   \widehat{\bit_i}(\alpha)
   =& \frac{1}{N} \left( \sum_{y=0}^{2^{i}-1} \omega_N^{-\alpha y} \right) \left(\sum_{b=0}^1 (-1)^b\omega_N^{-\alpha2^ib} \right) \left( \sum_{z=0}^{2^{k-i-1}-1} \omega_N^{-2^{i+1} \alpha z} \right)\\
   & = \frac{1}{N} \left( \sum_{y=0}^{2^{i}-1} \omega_N^{-\alpha y} \right) \left( 1 - \omega_N^{-2^i \alpha} \right) \left( \sum_{z=0}^{2^{k-i-1}-1} \omega_N^{-2^{i+1} \alpha z} \right) \, .
\end{split}
\end{equation*}
The third sum is just a sum over all $2^{k-i-1}$-th roots of unity, so it is $2^{k-i-1}$ when $\alpha$ is a multiple of $2^{k-i-1}$ and otherwise is zero.
The middle term $(1 - \omega_N^{-2^i \alpha})$ is therefore 2 when $\alpha$ is an odd multiple of $2^{k-i-1} $ and is zero if it is an even multiple.
For the first sum, we know from Lemma~\ref{lemma:weyl} that
\begin{equation}\label{eq:bits}
\Bigg|\sum_{y=0}^{2^{i}-1} \omega_N^{-\alpha y}\Bigg|
= \Bigg|\frac{\sin\left(\frac{\pi}{N}\alpha(2^{i}-1)\right)}{\sin\left(\frac{\pi}{N}\alpha\right)}\Bigg|
\leq \frac{1}{\big|\sin\left(\frac{\pi}{N}|\alpha|_N\right)\big|} \leq \frac{N}{\pi |\alpha|_N} \, .
\end{equation}

The result then follows:
the Fourier coefficient $\widehat{f}(\alpha)$ is zero when $\alpha$ is not an odd multiple of $2^{k-i-1}$ and when it is non-zero it has magnitude bounded by $2^{k-i}/|\alpha|_{2^k}$.
\end{proof}

\begin{corollary}
For every $i\in \N$, the $i$-th bit functions over $\Z_{2^k}$, i.e. $\{\bit_i:\Z_{2^k} \to \{-1,1\} \}_{k>i}$, are concentrated.
\end{corollary}

\begin{corollary}
For every $i\in \N$, the $i$-th bit functions $\{\bit_i:\Z_n \to \{-1,1\} \}_{n>2^i}$ are concentrated.
\end{corollary}

\paragraph*{Acknowledgements} We thank our supervisor Steven Galbraith for helpful assistance.

\appendix

\section{Tightness of the bound in Theorem \ref{thm:e-concentrated}}
We refine Example \ref{ex:switch_down} as follows.
Let $B:\Z_m \to \{1, -1\}$ be chosen uniformly at random and consider the function
\[ f(x) := \left\{ \begin{array}{ll} B(x) & \text{ when } 0 \le x < m , \\
     \chi_{\alpha,n}(x) & \text{ when } m \le x < n. \end{array} \right. \]
Then,
\[ \widehat{f}(\alpha+j) = \frac{1}{n}\sum_{x\in\Z_n}f(x)\overline{\chi_{\alpha+j}(x)} = \frac{1}{n}\sum_{0\leq x<m}B(x)\overline{\chi_{\alpha+j}(x)} + \frac{1}{n}\sum_{m\leq x<n}\omega_n^{-j x} \, . \]
One expects $\sum_{0\leq x<m}B(x)$ to be a random walk, and so with high probability $|\sum_{0\leq x<m}B(x)|<\sqrt{n/2}$, hence $|\frac{1}{n}\sum_{0\leq x<m}B(x)\overline{\chi_{\alpha+j}(x)}| < \frac{1}{\sqrt{2n}}$. The value $|\frac{1}{n}\sum_{m\leq x<n}\omega_n^{jx}|$ can be shown to be large for small odd $\pm j$, using Lemma \ref{lemma:weyl} (see also Eq. (\ref{eq:bits})), as well as for $j=0$.

For example the value $|\frac{1}{n}\sum_{m\leq x<n}\omega_n^{jx}|$ is $0.5$ for $j=0$ and using Lemma \ref{lemma:weyl}, it is about $0.636/2 = 0.318$ for $j=\pm1$. Thus, $|\widehat{f}(\alpha)| > 0.5 - \frac{1}{\sqrt{2n}}$ and $|\widehat{f}(\alpha \pm1)| > 0.318 - \frac{1}{\sqrt{2n}}$. Similarly, $|\widehat{f}(\alpha \pm3)| > 0.106 - \frac{1}{\sqrt{2n}}$ and $|\widehat{f}(\alpha \pm5)| > 0.063 - \frac{1}{\sqrt{2n}}$.

It is straightforward to show that given $\epsilon$, then for large enough $n$, we can approximate $f$ such that $\|f-f|_\Gamma\|_2^2 \leq 0.5 + \epsilon$, for a set $\Gamma$ of size that depends on $\epsilon^{-1}$. Specifically, $f$ is $(0.5 + \epsilon)$-concentrated for any $\epsilon>0$.

On the other hand, we have $\widetilde{f}(x) = B(x)$ and one expects this function to have no large coefficients (all coefficients bounded by $\frac{1}{\sqrt{m}}$ as shown above). We conclude that these large coefficients cancel most of each other's magnitude in the function $\widetilde{f}$.

From Theorem \ref{thm:e-concentrated} we know that $\widetilde{f}$ is $(2(0.5 + \epsilon)+\eta)$-concentrated for any $\eta>0$. Here the size of the set $\Gamma'$ on which $\widetilde{f}$ is concentrated depend on $\epsilon$ and on $\eta$. As we can set $\epsilon,\eta$ to be as small as we want, we have that $\widetilde{f}$ is $1$-concentrated, which trivially holds as $\widetilde{f}$ has norm $1$. Moreover, as it has no large coefficients at all, and as we show in the third example under Example \ref{ex:concentrted} above, with high probability it is not ($1-\epsilon'$)-concentrated for any $\epsilon'>0$. Therefore, this example shows that the bound given in Theorem \ref{thm:e-concentrated} is tight.


\section{Values for the algorithm}

Let $\Gamma\subset\Z_n$ be the set of all $\tau$-heavy coefficients of $f$, and suppose $\Gamma\neq\emptyset$.
Then $\| f - f|_\Gamma \|_2^2 = \|f\|_2^2 - \sum_{\alpha\in\Gamma}|\widehat{f}(\alpha)|^2 \leq \|f\|_2^2-\tau$.
By Proposition~\ref{prop:Gamma-concentrated}, there exists a set $\Gamma'\subset\Z_m$ with
\begin{equation}\label{eq:Gamma'}
|\Gamma'| \leq |\Gamma| (|\Gamma|\|f\|_2^2/\epsilon'+3)
\end{equation}
such that
\begin{equation}\label{eq:norm}
\|\widetilde{f} - \widetilde{f}|_{\Gamma'}\|_2^2 \leq \frac{n}{m}\|f\|_2^2 - \tau + \epsilon' + 2\sqrt{(1-\tau)\epsilon'} \, .
\end{equation}

We choose an appropriate $\epsilon'$ such that the right-hand side is non-trivial (for example, if $\tau < \frac{n}{m}\|f\|_2^2/2$ we can make the right-hand side equal $\frac{n}{m}\|f\|_2^2 - 2\tau$).
We calculate a suitable threshold $\tau'$ for the function $\widetilde{f}$.
Since $\|f\|_2^2 = \sum_{\alpha\in\Z_n}|\widehat{f}(\alpha)|^2$, then $f$ can have at most $\|f\|_2^2 / \tau$ coefficients larger than $\tau$. In other words,
\begin{equation}\label{eq:bound}
|\Gamma| \leq \|f\|_2^2 / \tau \, .
\end{equation}
From (\ref{eq:norm}) we have
\[ \|\widetilde{f}|_{\Gamma'}\|_2^2 = \|\widetilde{f}\|_2^2 - \|\widetilde{f} - \widetilde{f}|_{\Gamma'}\|_2^2  \geq \frac{n}{m}\|f\|_2^2-\frac{n}{m}\|f\|_2^2 + \tau - \epsilon' - 2\sqrt{(1-\tau)\epsilon'} = \tau - \epsilon' - 2\sqrt{(1-\tau)\epsilon'} \, . \]
It follows that in $\Gamma'$ there is at least one coefficient $\widehat{\widetilde{f_{\;}}}(\beta)$ such that
\[ |\widehat{\widetilde{f_{\;}}}(\beta)|^2 \geq \frac{\tau - \epsilon' - 2\sqrt{(1-\tau)\epsilon'}}{|\Gamma'|} \geq \frac{\tau - \epsilon' - 2\sqrt{(1-\tau)\epsilon'}}{(\|f\|_2^2 / \tau)\left((\|f\|_2^2 / \tau)\|f\|_2^2 / \epsilon' + 3\right)} \, , \]
using (\ref{eq:Gamma'}) and (\ref{eq:bound}).
We set $\tau'$ to be the latter value.

\end{document}